\documentclass[11pt,a4paper,aps]{article}
\usepackage{etex}
\usepackage{inputenc}
\usepackage[T1]{fontenc}
\usepackage{lmodern}					
\usepackage{textcomp}				
\usepackage{verbatim}				
\usepackage[english]{babel}
\usepackage[toc,page]{appendix} 		
\usepackage{cite}
\usepackage{amsmath,					
			amsthm,
			amsfonts,
			amssymb,
			mathtools}
\usepackage{enumitem} 

\usepackage[all]{xy}
\usepackage{tikz-cd} 
\usetikzlibrary{tikzmark}
\usetikzlibrary{arrows}

\usepackage{graphicx}				
\usepackage{tabularx}				

\usepackage{authblk}   

\definecolor{pur}{RGB}{255,102,102}

\usepackage[hmargin=2.2cm,vmargin=2.5cm]{geometry}		
\numberwithin{equation}{section}

\usepackage[colorlinks]{hyperref}			

\newcommand{\dd}{\mathrm{d}}

\newtheorem{definition}{Definition}
\newtheorem*{definition*}{Definition}
\newtheorem{theoreme}{Theorem}
\newtheorem*{theoreme*}{Theorem}

\newtheorem*{lemme*}{Lemma}

\newtheorem*{conjecture*}{Conjecture}

\theoremstyle{remark}
\newtheorem*{remarque}{Remark}
\newtheorem{example}{Example}

\newtheorem*{prooftheo*}{Proof of the Theorem}

\title{A short guide through integration theorems of generalized distributions}

\author[]{Sylvain Lavau\thanks{sylvain.lavau@ens-lyon.fr}}
\affil[]{{\small \textit{Max Planck Institute for Mathematics}, Bonn, Germany.}}
\date{}

\begin{document}

\maketitle


\abstract{The generalization of Frobenius' theorem to foliations with singularities is usually attributed to Stefan and Sussmann, for their simultaneous discovery around 1973. However, their result is often referred to without caring much on the precise statement, as some sort of magic spell. This may be explained by the fact that the literature is not consensual on a unique formulation of the theorem, and because the history of the research leading to this result has been flawed by many claims that turned to be refuted some years later. This, together with the difficulty of doing proof-reading on this topic, brought much confusion about the precise statement of Stefan-Sussmann's theorem. 
This paper is dedicated to bring some light on this subject, by investigating the different statements and arguments that were put forward in geometric control theory between 1962 and 1994 regarding the problem of integrability of generalized distributions. We will present the genealogy of the main ideas and show that  many mathematicians that were involved in this field made some mistakes that were successfully refuted. Moreover, we want to address the prominent influence of Hermann on this topic, as well as the fact that some statements of Stefan and Sussmann  turned out to be wrong. 
In this paper, 
we intend to provide the reader with a deeper understanding of the problem of integrability of generalized distributions, and to reduce the confusion surrounding these difficult questions.}




\section{Introduction}

Foliation theory is the study of foliations on manifolds. A \emph{foliation} on a manifold $M$ is a partition of $M$ into connected immersed submanifolds, that are called \emph{leaves}. A foliation is called \emph{regular} if the leaves have the same dimension, and \emph{singular} otherwise. Over every point $x\in M$, the tangent space of the leaf $L_x$ through $x$ is a subspace of the tangent space of $M$. The data of a subspace $\mathcal{D}_x$ of $T_xM$ at every point $x\in M$ define what is called a \emph{distribution}  $\mathcal{D}=\bigcup_{x\in M}\mathcal{D}_x$ on $M$. Notice that a distribution is not necessarily a sub-bundle of $TM$ because it may not have constant rank. 
For example, for a regular foliation, since the leaves have the same dimension, the induced distribution formed by the tangent spaces at every point has constant rank over $M$. In the singular case however, the dimension of the tangent spaces to the leaves may vary from leaf to leaf. Since the tangent spaces to a given foliation form a distribution $\mathcal{D}$ on $M$, and since the space of vector fields tangent to the leaves are closed under Lie bracket, then $\mathcal{D}$ inherits the Lie bracket of vector fields. More precisely, we say that a distribution $\mathcal{D}$ is \emph{involutive} if for every two sections $X,Y$ of $\mathcal{D}$, the commutator $[X,Y]$ is a section of $\mathcal{D}$ as well. On the other hand, 
a given distribution $\mathcal{D}$ may not come from the tangent spaces of a foliation. Then, we say that $\mathcal{D}$ is \emph{integrable} if there exists a foliation such that each leaf $L$ satisfies $T_xL=\mathcal{D}_x$ for every $x\in L$. A legitimate question is thus: `Given a distribution on $M$, what are the conditions under which it is integrable to a foliation?'

This question is a modern formulation of a set of results and investigations that were related to -- but not directly concerned with -- the topic of integrating distributions into foliations.
Originally, the problem emerged as finding the solutions of non-linear first-order partial differential equations, and was pioneered by Lagrange that provided a method for systems involving up to two independent variables. It was then formalized to an arbitrary number of variables by Pfaff in his memoir at the University of Berlin in 1815, hence the name of \emph{Pfaffian systems} \cite{Hawkins}. He showed how one may transform a set of $n$ first-order non-linear partial differential equations into a set of $2n$ ordinary linear differential equations. The  simplification method presented by Pfaff could be seen as finding a submanifold of the space of variables on which some specific one-form vanishes. The problem was that Pfaff could not make precise what were the conditions under which one could use this simplification. This question -- designated as the \emph{problem of Pfaff} -- led to multiple investigations that finally found an accurate answer by Frobenius in 1877. 

Actually, the name `Frobenius' theorem' comes from Cartan in 1922 because Frobenius' result had an tremendous influence on Cartan's calculus of differential forms. Frobenius' paper is actually archetypal of the production of the Berlin school of Mathematics at that time, which promoted the idea that a clear, rigorous and systematic presentation of the arguments was just as important as the discovery of new results by whatever means. Frobenius and his contemporaries in Berlin participated in a shift of paradigm in modern mathematics by improving standards of rigor and presentation \cite{Hawkins}. This is in part the reason why Frobenius is remembered for this theorem, whereas the work of his predecessors has been forgotten.

Indeed, it turns out that Frobenius' theorem is actually an algebraic reformulation of a result published in 1840 by Deahna \cite{Deahna}, who then became a teacher in a secondary school (that was common at the time) before his premature death at age 28 in 1844. Deahna's work did not gain much interest, and it was later Clebsch in 1861, editing a posthumous article of Jacobi, who improved Pfaff's argument \cite{Clebsch}. Even if the problem of solving Pfaffian systems had been around for many years, it was the article of Clebsch which motivated the interest of Frobenius on this question. Simultaneously, unaware of Clebsch's investigations, Natani proposed another approach to the question of solving Pfaffian systems, but the relationship with Clebsch's work was not realized before a few years \cite{Hawkins}. The modern formulation of Frobenius' theorem does not correspond to the one appearing in its original paper \cite{Frobenius}, because it has been modified to fit with modern-day standards and conventions:
\begin{theoreme}\emph{\textbf{Frobenius (1877)}}
Let  $M$ be a smooth manifold and let $\mathcal{D}$ be a smooth distribution of constant rank on $M$. Then $\mathcal{D}$ is integrable into a regular foliation if and only if $\mathcal{D}$ is involutive.
\end{theoreme}
\noindent  Involutivity is a natural necessary condition because the set of vector fields on any leaf of a foliation is involutive, hence the corresponding distribution should be as well. Frobenius implicitly proved that it turns out to be a sufficient condition in the regular case, see \cite{Lee} for a short proof.

\bigskip
Interestingly, the problem of integrating generalized distributions was approached in the same way as for the Frobenius' theorem: i.e. solving a set of linear differential equations. Indeed, at the turn of the 1950s, numerous investigations in the field of control theory $-$ the study of the solvability of first-order differential equations under the influence of one or more external parameters $-$ arose and developed in the following years. Unfortunately, the picture in control theory involves external parameters that modify and generalize the structure of Pfaffian system, so those parameters prevent to use the integrability arguments of Pfaff, Deahna, Clebsch and Frobenius.
 
 Inspired by the work of Carath\'eodory on the geometrization of the calculus of variations and Pfaffian systems, many mathematicians aimed at solving some linear differential systems from a geometric perspective \cite{jurdjevic}. Chronologically, 
Hermann was the first  to draw a bridge between control theory and differential geometry in 1963 \cite{Hermann1963}. In his view, the solutions of a differential system would correspond as the \emph{attainable set} of points that are reachable from the initial data, following the flow of the vector fields associated to the differential equations. Hence investigating integrability conditions of generalized distributions into singular foliations appeared as a necessity for control theorists.
Actually, it was also Hermann who stated the first integrability conditions both in the smooth and the analytic cases (without proof for this last case)\cite{Hermann1963}. Nagano proved that analyticity together with involutivity are sufficient conditions for integrability in 1966 \cite{Nagano}. 
Then, after a few years of small improvements \cite{matsuda1968,Lobry70}, Stefan and Sussmann independently clarified in 1973 the conditions for a family of smooth vector fields $F$ to induce an integrable distribution \cite{Sussmannofficiel, Stefanofficiel}: the only assumption is that the induced distribution $\mathcal{D}^F$ has to be invariant under the action of the flow of any element of $F$. Of course, both of them had supplementary material in their respective and quasi-simultaneous papers, but this is the main result that they had in common and that was  thoroughly used in control theory. The $F$-invariance was not in fact a new idea, since it was around since the first proposals of Hermann in 1963 \cite{Hermann1963}, and since it was made explicit by Lobry in 1970 \cite{Lobry70}. The breakthrough of Stefan and Sussmann was showing that one can drop Hermann's assumption that $F$  is a sub-Lie algebra of the space of globally defined vector fields.
After this, Stefan himself deepened his research on the integrability problem and made interesting discoveries on this topic \cite{stefan1980}.



This paper is an investigation of the different statements and arguments that occurred in geometric control theory between 1962 and 1994 related to the problem of integrating generalized distributions. 
We will see  that, even if some hard work was done, many results were forgotten or mistakenly attributed to other mathematicians, and  above all, that many people involved in this story made mistakes that led to some confusion that persists today. 
In particular we want to address the persistent claim that Stefan's and Sussmann's results were not totally correct when they were published, and that Balan corrected them. We will see to what extent this is true and we will give precisions on Balan's statements.
The goal of this article is to clarify who said what and what is proven regarding these subtle questions. In Section \ref{background}, we recall mathematical notions that are commonly used in the field. In Section \ref{nagano}, we present the most historical and useful results that were proven in the 1960s, namely Nagano's and Hermann's theorems. In Section \ref{lobry}, we give an overview of the path that led from these pioneers to the well-known theorem of Stefan and Sussmann in 1973. Then, in Section \ref{stefan}, we discuss improvements and some results that followed this breakthrough. In Section \ref{conclusion} -- the conclusion -- we acknowledge the history of this long standing question, we clarify present debates and we propose future research.

\subsection*{Acknowledgments}

I am indebted toward Arjen Baarsma, whose patience and benevolence convinced me to write this paper. His accurate insights were very welcome and had a decisive influence on this work. I also want to thank Alexis Laignelet for a careful rereading of this paper. This research was partially supported by CMUP (UID/MAT/00144/2013) funded by FCT (Portugal) with national funds. I sincerely thank Peter Gothen and Andr\'e Oliveira for their kindness and for the opportunity they gave me to do my own research at CMUP.

\section{Mathematical background}\label{background}

There are two approaches to the problem of integrating a distribution into a foliation. The approach in control theory  starts from a linear differential system, then defines a set of vector fields that carries all the information from the differential equations, and then look for solutions of these equations as the points that are reachable by the flows of these vector fields. On the contrary, the geometric approach is more focused on the concept of distribution as a given object, and questions the possibility that this distribution is the tangent space of a foliation. Thus, it is not surprising that the two communities refer to the same theorems, but under different names and formulations. Let us now recall some fundamental mathematical notions:

\begin{definition}
On a smooth manifold $M$, a \emph{(generalized) distribution} $\mathcal{D}$ is the assignment, to each point $x\in M$, of a subspace $\mathcal{D}_x$ of the tangent space $T_xM$.
\end{definition}

\noindent A distribution $\mathcal{D}$ is \emph{smooth} at a point $x$ if any tangent vector $X(x)\in \mathcal{D}_x$ can be locally extended to a smooth vector field $X$ on some open set $U\subset M$ such that $X(y)\in\mathcal{D}_y$  for every $y\in U$.  The space of smooth sections of $\mathcal{D}$ is the sub-sheaf $\Gamma(\mathcal{D}):U\mapsto \Gamma_U(\mathcal{D})$ of the sheaf of vector fields $\mathfrak{X}$ consisting of smooth (resp. analytic) vector fields that take values in $\mathcal{D}$.
As a side remark, let $x\in M$, then any family of independent sections of $\mathcal{D}$ that span $\mathcal{D}_x$ is locally free. 
It implies that the rank of the distribution in a neighborhood of $x$ is greater than or equal to the dimension of $\mathcal{D}_x$.  All these definitions have similar counterparts in the real analytic category, i.e. when all objects are analytic.


 In the 1960s, mathematicians used mostly globally defined vector fields since they had in mind the link between geometry and control theory. In the 1970s, Stefan and Sussmann gave the first geometric results that involve locally defined vector fields. When it is not specified, vector fields can be either globally or locally defined.
A set of (possibly locally defined) vector fields $F$ induces a distribution $\mathcal{D}^F$ on $M$ by the formula: 
\begin{equation*}\mathcal{D}^F_x=\mathrm{Span}\big(X(x)\, |\, X\in F\big)\end{equation*} 
for every $x\in M$. Also, $F$ induces a pseudogroup of (possibly local) diffeomorphisms of $M$  \cite{Sussmannofficiel, Stefanofficiel, bullo}. First, any $X\in F$ defines a flow $t\mapsto\phi^X_t$: for every $t\in\mathbb{R}$, the map $\phi^X_t$ is a (local) diffeomorphism of $M$, with inverse $\phi^{X}_{-t}$. The set of all (local) diffeomorphisms $\big\{\phi^X_t\big\}_{t\in\mathbb{R}}$ is thus a group that is called the \emph{group of diffeomorphisms generated by $X$}, and it is denoted by $G^X$. Second, the set of all flows $\big\{\phi^X_t\,|\,X\in F, t\in \mathbb{R}\big\}$ generates a subgroup of the group of (local) diffeomorphisms of $M$, which is the smallest group generated by $\bigcup_{X\in F}G^X$. It is called the \emph{group of diffeomorphisms generated by $F$}, and it is denoted by $G^F$. An element of $G^F$ is a composition of flows of vector fields:
\begin{equation}
\phi^{X_n}_{t_n}\circ\ldots\circ\phi^{X_2}_{t_2}\circ\phi^{X_1}_{t_1}
\end{equation}
where $X_i\in F$ and $t_i\in\mathbb{R}$ for every $1\leq i\leq n$.
 
 \begin{definition}
 Let $M$ be a smooth manifold and let $F$ be a family of vector fields.
 Given a point $x\in M$, the set of all points that can be reached from $x$ by using the elements of $G^F$ is called the \emph{$F$-orbit of $x$}, and is denoted by $O^F_{x}$.
 
  We say that a distribution $\mathcal{D}$ is \emph{$F$-invariant} if it is invariant under the action of $G^F$, i.e. if for any $X\in F$ we have:
\begin{equation*}
(\phi^X_t)_\ast\big(\mathcal{D}_y\big)\subset\mathcal{D}_{\phi^X_t(y)}
\end{equation*}
for every $y$ in the domain of $X$ and $t\in \mathbb{R}$.

\end{definition}

  The link between $F$-orbits and the integration of distributions is subtle. The point is that the distribution $\mathcal{D}^F$ generated by the family $F$ may not be equal to the tangent space of the $F$-orbits. Indeed, by definition of the Lie bracket, the $F$-orbits of a given linear differential system contain the integral curves of the commutators of vector fields of $F$. Sussmann provides some precision in \cite{Sussmannofficiel}: given some $x\in M$, if $X,Y$ are tangent vectors to the orbit $O^F_{x}$ at $x$, then $[X,Y]$ is tangent to $O^F_{x}$ as well. However, the distribution $\mathcal{D}^F$ may not be closed under Lie bracket because $F$ may not be either. This implies that in general we have $\mathcal{D}^F_x\subset T_xO^F_{x}$, with a strict inclusion. Hence, for control theorists, this is not very interesting to look at the distribution $\mathcal{D}^F$, but rather to the distribution that contains also the directions spanned by the commutators of elements of $F$. To make things more precise, let us define the \emph{Lie closure of $F$} as the smallest Lie algebra generated by elements of $F$, and denote it by $\mathrm{Lie}(F)$. It is the smallest space of vector fields such that $[S,\mathrm{Lie}(F)]\subset\mathrm{Lie}(F)$. Then the preceding argument implies that the $F$-orbits contain the $\mathrm{Lie}(F)$-orbits, then as a corollary  they coincide. The problem of finding the solutions of a linear differential system could then be reformulated as integrating the distribution induced by $\mathrm{Lie}(F)$. This is consistent with the idea that the space of vector fields tangent to the orbits are closed under Lie bracket. In general, if the $F$-orbits are submanifolds, one has:
  \begin{equation*}\mathcal{D}^F_y\subset\mathcal{D}^{\mathrm{Lie}(F)}_y\subset T_yO^F_{x}\end{equation*}
 for every  $y\in O^F_{x}$.
The equality on the right hand side is guaranteed when $F$ satisfies some particular conditions. For example, Nagano showed  that if $F$ is analytic, the equality is automatically satisfied, whereas  Hermann's condition is  that $F$ be locally finitely generated. Essentially, these are the two cases that most control theorists consider, see \cite{bullo, jurdjevic, Agrachev} and Section \ref{nagano} for details. On the other hand, Stefan and Sussmann proved that $\mathcal{D}^F_y= T_yO^F_{x}$ at every point $x\in M$ and $y\in O^F_{x}$ if the distribution $\mathcal{D}^F$ is $F$-invariant \cite{Sussmannofficiel, Stefanofficiel}.

  \begin{example}
 On $M=\mathbb{R}^3$, let $F$ be the family of vector fields generated by the action of $\mathfrak{so}(3)$ on $\mathbb{R}^3$. It defines an integrable distribution: the leaves are concentric spheres, and the point at the origin. The tangent bundle of each sphere is indeed invariant under the action of $\mathfrak{so}(3)$.
  \end{example}

\begin{remarque}
 If we restrict ourselves to the semi-group $H^F$ generated by the flows of elements of $F$ with positive times only, we may obtain a different set that we call the \emph{attainable set of $x$}. This correspond to the situation where only forward-in-time motions are allowed, and this is essentially the set that control theorists are interested in.   These are the conventions mostly used by Sussmann and control theorists \cite{Agrachev, Sussmannofficiel, jurdjevic, bullo}. Notice that Stefan's conventions are slightly different: in his fundamental paper \cite{Stefanofficiel}, he designates the $F$-orbits as accessible sets. 
 An equivalent formulation is made by using compositions of integral curves of elements of $F$: an \emph{integral path} is a piecewise smooth path $\gamma:[a,b]\to M$ such that for every open interval $I\subset[a,b]$ where $\gamma$ is differentiable, there exists a vector field $X\in F$ such that:
\begin{equation}
\frac{\dd}{\dd t}\gamma (t)=X\big(\gamma(t)\big)\label{integralpath}
\end{equation}
Then the attainable set of $x$ is precisely the set of points reachable by such integral paths.
 Obviously, if the family $F$ is \emph{symmetric}, i.e. if $F=-F$, then the $F$-orbit of $x$ and the attainable set of $x$ coincide. 
\end{remarque}

This is now time to introduce more geometrical tools, and where we turn to the theory of integration of distributions:
\begin{definition}
Given a distribution $\mathcal{D}$, an \emph{integral manifold of  $\mathcal{D}$} is a connected immersed submanifold $N\subset M$ such that $T_yN=\mathcal{D}_y$ for every $y\in N$.
\end{definition}
\noindent Given a point $x\in M$, we say that the distribution $\mathcal{D}$ is \emph{integrable at $x$} if there is an integral manifold $N$ of $\mathcal{D}$ that contains $x$. 
An integral manifold through $x$ is said \emph{maximal} if it contains every integral manifolds through $x$. A distribution $\mathcal{D}$ is \emph{integrable} if for every $x\in M$, there exists a maximal integral manifold through $x$. In particular, if $\mathcal{D}$ is integrable, $M$ is the disjoint union of the maximal integral manifolds of $\mathcal{D}$.
Moreover, if an integrable distribution $\mathcal{D}=\mathcal{D}^F$ is induced by some family of vector fields $F$, then the maximal integral manifolds are the $F$-orbits, this is precisely the content of Stefan-Sussmann's theorem, see Section \ref{lobry}.

Here, the word `integrable' refers directly to the theory of foliations. Recall that one defines a (possibly singular) \emph{foliation} as a partition of $M$ into connected immersed submanifolds, that are called the \emph{leaves of the foliation}.  
These definitions imply that if a distribution $\mathcal{D}$ is integrable, then the maximal integral manifolds of $\mathcal{D}$ form the leaves of a foliation. Stefan has even shown that there exist \emph{distinguished charts} that are adapted to the foliation, see Section \ref{conclusion}. 
Given a point $x\in M$, we write the maximal integral manifold of $\mathcal{D}$ through $x$ as $L_x$ and we call it the \emph{leaf through $x$}. 
Since the map $x\mapsto\mathrm{dim}(L_{x})$ which associates to any point $x$ the dimension of its leaf is lower semi-continuous, the dimensions of the leaves in a neighborhood of $x$ are necessarily greater than or equal to $\mathrm{dim}(L_x)$. This is consistent with the fact that the rank of a distribution is  lower semi-continuous as well.
 A point $x\in M$ is said to be a \emph{regular point} if the dimension of the leaves is constant in some neighborhood of $x$, and a \emph{singular point} (or \emph{singularities}) otherwise. A leaf $L$ is said \emph{regular} if every point of $L$ is a regular point, and it is said \emph{singular} otherwise. The set of regular points is open and dense in $M$, and the leaves of highest dimensions are necessarily regular. 

\begin{example}
Let $\mathcal{D}$ be the smooth distribution on $\mathbb{R}^{2}$ defined by:
\begin{equation*}
\mathcal{D}_{(x,y)}=
\begin{cases}\mathrm{Span}\big(\frac{\partial}{\partial x}\big) \quad \text{for}\quad 0< y\\
\{0\} \hspace{1.32cm} \text{for}\quad y\leq 0
\end{cases}
\end{equation*}
The regular points are those that do not belong to the horizontal axis. This distribution is integrable into a foliation of $\mathbb{R}^{2}$ that has horizontal leaves for $y>0$ and points otherwise.
\end{example}

Last but not least, every theorem that are presented in this article are constructive, which means that they provide a recipe to build the integral manifolds of a distribution. In most cases, the integral manifolds are the $F$-orbits, but it does not come with a natural topology and smooth structure. This is where a result by Chevalley is systematically referred to: the construction of a `strong' topology on $M$ that is adapted to such integral manifolds. This construction is presented in chapter 3, Section VIII, of \cite{chevalley}.  More precisely, 
the idea is to define, for every $F$-orbit, a family of small patches that cover the orbit. In most cases, this is done by using the exponential map, because the integral curves of an element of $F$ stays in the orbit. Hence each point of an $F$-orbit would induce such a small patch in its vicinity, that is also an integral manifold through $x$. The union of all these patches is taken as a basis for the new topology, which turns to be finer than the original one. Mathematically, it is as if every open set in the older topology was now `foliated' by the integral manifolds through each of its points. This topology enables to rigorously define continuous maps, local homeomorphisms and so on. Then, one can rely on this topology to provide each leaf with a smooth (or real analytic) manifold structure. Originally, the construction of Chevalley was fit for analytic regular foliations, but Hermann and Nagano could adapt it easily to their context \cite{Agrachev}. They indeed only considered vector subspaces of $\mathfrak{X}(M)$, hence the basis of the new topology could be obtained from the small exponential patches based at each point. On the contrary, Stefan and Sussmann considered families of vector fields that may not satisfy the vector space axioms, and then they had to adapt the construction of Chevalley to their own needs. This is precisely for this reason that their proofs are tedious to go through, and that their work has to be acknowledged.

\begin{remarque}
As a final remark, a recent and important result \cite{texas} shows that any smooth distribution $\mathcal{D}$ is actually point-wise finitely generated, i.e. there is a finite family of vector fields $F$ such that $\mathcal{D}=\mathcal{D}^F$. However,  this fact does not imply that the sheaf of sections $\Gamma(\mathcal{D})$ (or any other family of vector fields generating $\mathcal{D}$) is finitely generated. For example, take the vector field $X=\chi(x)\frac{\partial}{\partial x}$ defined on $M=\mathbb{R}$, where the function $\chi$ is defined by:
\begin{equation*}
\chi=
\begin{cases}e^{-\frac{1}{x}} \hspace{0.47cm} \text{for}\quad  x>0\\
\hspace{0.23cm}0 \hspace{0.7cm} \text{for}\quad x\leq 0
\end{cases}
\end{equation*}
The associated distribution $\mathcal{D}^X$ consists in the null vector space on $\mathbb{R}_{-}$ and the tangent space $T_x\mathbb{R}$ on the open line $\mathbb{R}_{+}^*$. This distribution is point-wise generated by $X$: $\mathcal{D}^X_x=\mathrm{Span}\big(X(x)\big)$, and it is obviously integrable, as an integral curve of the vector field $X$. However the sections of $\mathcal{D}^X$ are not finitely generated in any neighborhood of 0, see \cite{texas}.
\end{remarque}

\section{Nagano and Hermann}\label{nagano}

As was said in the introduction, Hermann played a prominent role in the search for a generalization of Frobenius' theorem to generalized distributions.
In a paper in 1963 \cite{Hermann1963}, he made explicit the relationship between control theory, Pfaffian systems and foliation theory. Hermann introduces the geometric setup for control theory: linear differential systems can be equivalently seen as a family of vector fields. As such, he can be seen as the founder of geometric control theory.
In the same paper, he also gives not only one sufficient condition to the integrability problem, but he actually gave three. Two of them relate to the smooth case and were proven by himself one year earlier \cite{Hermann1962}. The last one, in the analytic case, is a claim that he had not yet proven at that time, and that was proven by Nagano in 1966 \cite{Nagano}.


Let us start with the analytic case. Item $(c)$ in \cite{Hermann1963} states that any analytic family of vector fields that is involutive induces an integrable distribution.
Working in the analytic category is simpler than in the smooth category because one can rely on some properties of analytic geometry. Reproducing the statement of Hermann, Nagano considers the set of globally defined analytic vector fields $\mathfrak{X}(M)$ as an infinite-dimensional Lie algebra, then he picks up a subspace $F$ (i.e. a set of globally defined vector fields) 
and his claim is as follows:

\begin{theoreme}\emph{\textbf{Nagano (1966)}}\label{theonagano}
Let $M$ be a real analytic manifold, and let $F$ be a sub-Lie algebra of $\mathfrak{X}(M)$. Then the induced analytic distribution $\mathcal{D}^F$ is integrable.
\end{theoreme}


%


\begin{proof} We only give here a sketch of the proof, and we refer to \cite{Nagano} for more details.
The original proof of Nagano consists in showing that for any point $x\in M$: 1. the set $\mathfrak{N}_x$ of integral manifolds through $x$ is not empty, and that 2. any finite intersection of integral manifolds through $x$, restricted to some open neighborhood of $x$, is an embedded integral manifold. He proves the first item by splitting the distribution in a neighborhood $U$ of $x$, in the sense that he selects a $\mathrm{dim}\big(\mathcal{D}^F_x\big)$-dimensional subspace $F(x)\subset F$ that generates the distribution at $x$, and complete it by a $\mathcal{O}(U)$-module $\mathcal{G}(x)$ that contains all  the vector fields of $F$ that are not in the point-wise span of $F(x)$. Here, $\mathcal{O}(U)$ means the ring of real analytic functions on $U$. Then he observes that the distribution $\mathcal{D}^F$ coincides with the distribution induced by $F(x)\oplus\mathcal{G}(x)$. Nagano then defines $N_x$ as the embedded submanifold induced by the exponential map $\exp:B_\epsilon\to M$, where  $B_\epsilon$ is a small ball of radius $\epsilon$ centered on zero in $F(x)$. He then uses the splitting to show that the embedded submanifold $N_x$ is an integral manifold of $\mathcal{D}^F$.

Then, he defines the set $L_x$ as the union of all integral manifolds through $x$: $L_x=\bigcup_{N\in\mathfrak{N}_x}N$. By construction the rank of $\mathcal{D}^F$ is constant over $L_x$, hence this is a good candidate for the leaf through $x$. The topology and the analytic structure on $L_x$ are inherited from the respective topologies and analytic structures of all $N$ in $\mathfrak{N}_x$. Open sets of $N\in\mathfrak{N}_x$ are considered to be open in $L_x$, and item 2. is crucial to show that a finite intersection of open sets is open, as a neighborhood of each of its points. The details of the construction of the topology and the atlas on $L_x$ can be found in Chevalley's book \cite{chevalley}. Chevalley proves Frobenius' theorem for analytic regular distributions, and Nagano's theorem is a direct generalization of this theorem to analytic singular distributions.
\end{proof}

The proof of Nagano uses the analyticity of the vector fields by summoning the property that any real function whose successive derivatives vanish at the origin is the zero function. This theorem generalizes Frobenius' theorem in a straightforward way to the analytic (and singular) case, because it doesn't assume anything other than involutivity. We see in the following example that in the smooth case, the involutivity condition is not sufficient anymore:

\begin{example}\label{ex:nagano}
Let $\mathcal{D}$ be the smooth distribution on $\mathbb{R}^{2}$ defined by:
\begin{equation*}
\mathcal{D}_{(x,y)}=
\begin{cases}T_{(x,y)}\mathbb{R}^{2} \hspace{0.8cm} \text{for}\quad 0< x\\
\mathrm{Span}\big(\frac{\partial}{\partial x}\big) \hspace{0.55cm} \text{for}\quad x\leq 0
\end{cases}
\end{equation*}
where we understand $\langle\frac{\partial}{\partial x}\rangle$ as the subspace of $T_{(x,y)}\mathbb{R}^{2}$ spanned by the tangent vector $\frac{\partial}{\partial x}$. Sections of this distribution consists of sums of horizontal vector fields and vertical vectors fields which vanish for $x\leq0$. The bracket will preserve this property and therefore the distribution is involutive.

We now show that though this smooth distribution is involutive, it cannot be integrated into a singular foliation. On the right half-plane (for $x>0$), the leaf associated to this distribution is all of the open half-plane. On the contrary, on the open left half-plane (for $x< 0$) the vertical vector field vanishes hence the distribution admits integral manifolds that are horizontal lines (since at each point the vector field $\frac{\partial}{\partial x}$ generates the tangent space to the leaf). The maximal integral manifold passing through the point $(x,y)$ (for $x\leq0$) is the line $N_y=\big\{(w,y)\,|\,w<0\big\}$. On the vertical axis, the distribution is spanned by $\frac{\partial}{\partial x}$ but for any given $y\in\mathbb{R}^2$, the subset $N_y\cup \{(0,y)\}$ is not an immersed submanifold of $\mathbb{R}^2$, because it is not open on its right end. Hence the points on the vertical axis do not admit maximal integral manifolds, i.e the distribution is not integrable.
\end{example}

The above example shows that in the smooth case, involutivity does not imply integrability. Hermann proposed two conditions to solve this issue. The condition for which Hermann is known is condition $(b)$ in \cite{Hermann1963} and corresponds to the condition found in the theorem now bearing his name that was proven one year earlier in \cite{Hermann1962}.  Because he is focused on the relationship with control theory where equations may be defined everywhere, Hermann considers only globally defined vector fields. In other words he relies on subspaces $F\subset\mathfrak{X}(M)$ to describe a linear differential system of equations.
 He says that $F\subset\mathfrak{X}(M)$ is \emph{locally finitely generated} if for every open set $U\subset M$, there exists $X_1,\ldots,X_p\in F$ such that the restriction of $F$ to $U$ is contained in the $\mathcal{C}^\infty(U)$-module generated by $X_1,\ldots, X_p$. In other words: $F|_U\subset\mathcal{C}^\infty(U)\mathrm{Span}\big(X_1|_U,\ldots,X_p|_U\big)$. Then, Hermann's statement in \cite{Hermann1962} is:

\begin{theoreme}\emph{\textbf{Hermann (1962)}}\label{theohermann}
Let $M$ be a smooth manifold, and let $F$ be a locally finitely generated sub-Lie algebra of $\mathfrak{X}(M)$. Then the induced smooth distribution $\mathcal{D}^F$ is integrable.
\end{theoreme}

\begin{proof}
As before, this is but a sketch of the original proof, whose details can be found in \cite{Hermann1962}.
The proof of Hermann relies on showing that the rank of the distribution is constant along the integral curve of any vector field $X\in F$, and hence on the $F$-orbits. Hermann shows this result by using the fact that $F$ is involutive, and since it is also locally finitely generated, the Lie bracket with $X$ can be expressed in terms of local generators of the distribution. He uses this property to obtain a matricial differential equation in $T_xM$, and solving it shows that the rank of $\mathcal{D}^F$ is locally constant on the integral curve of $X$. By a compactness argument, he concludes that it is constant on the entire integral curve of $X$. 

He then defines $L_x$ as the set of all points of $M$ that can be joined to $x$ by an integral path of $F$. This set of points $L_x$ coincides with the $F$-orbit of $x$ because the set of vector fields $F$ is symmetric. Notice that the involutivity of $F$ implies that $\mathcal{D}^F=\mathcal{D}^{\mathrm{Lie}(F)}$.
Since the rank of $\mathcal{D}^F$ is constant along the integral curve of any element $X\in F$, it is constant over $L_x$, this is then a good candidate to be the leaf of $\mathcal{D}^F$ through the point $x$. 

The topology and the smooth atlas on $L_x$ are induced by the construction of the leaf itself: for any point $y\in L_x$, one can find a subspace $F(y)\subset F$ whose dimension is the dimension of $L_x$ (hence the importance of showing that is it constant over $L_x$), and then the exponential map defines an embedding of a neighborhood of zero in $F(y)$ into $M$ such that $0$ is mapped on $y$. By definition the image $N_y$ of this exponential map is entirely contained in $L_x$.
 Then Hermann uses these embedded submanifolds $\{N_y\}_{y\in M}$ as a basis for the new topology on $M$, as discussed in the construction of Chevalley \cite{chevalley}. 
 This topology is used afterwards to equip the $F$-orbits with a manifold structure.\end{proof}
 
 All this discussion is made possible because the families of vector fields that Hermann studies are subspaces of $\mathfrak{X}(M)$, thus he can use the exponential map  as a tool to generate charts on the $F$-orbits. This is definitely not allowed anymore in the generalization of this result by Stefan and Sussmann, who work with families of vector fields that are not necessarily vector spaces.
 Notice that an alternative choice of charts for $L_x$ is made by Lobry in \cite{Lobry70}, who provides a set of `curviligne coordinates' adapted to any choice of basis of $\mathcal{D}^F_y$. 
 
The last integrability condition proposed by Hermann in his 1963 paper is in fact nothing but the second part of the proof of Theorem \ref{theohermann}, see condition $(a)$ in \cite{Hermann1963}. 
\noindent 
More precisely, Hermann's statement is that if $F$ is a sub-Lie algebra of $\mathfrak{X}(M)$, and if the rank of the distribution $\mathcal{D}^F$ is constant on the integral paths of $F$, then $\mathcal{D}^F$ is integrable.
Notice that the converse of Hermann's statement is not true, even though it is claimed in Theorem 1.41 in \cite{Olver}. It can indeed be refuted by the following counter-example due to Balan in some unpublished notes:
\begin{example}
On $M=\mathbb{R}^2$, let $X=\varphi(x,y)\frac{\partial}{\partial x}$ and $Y=\big(x^2+y^2\big)\frac{\partial}{\partial y}$, where:
\begin{equation*}
\varphi(x,y)=
\begin{cases}e^{-\frac{1}{x^2+y^2}} \hspace{0.42cm} \text{for}\quad  (x,y)\neq(0,0)\\
\hspace{0.55cm}0 \hspace{0.9cm} \text{for}\quad (x,y)= (0,0)
\end{cases}
\end{equation*}
Let $F$ be the $\mathcal{C}^\infty(M)$-module generated by $X$ and $Y$. The induced distribution $\mathcal{D}^F$ is given by:
\begin{equation*}
\mathcal{D}_{(x,y)}=
\begin{cases}T_{(x,y)}\mathbb{R}^{2} \hspace{0.48cm} \text{for}\quad (x,y)\neq (0,0)\\
\hspace{0.5cm}0 \hspace{1.05cm} \text{for}\quad (x,y)= (0,0)
\end{cases}
\end{equation*}
which is obviously integrable. However, the commutator $[X,Y]$ is, for any couple $(x,y)\neq(0,0)$:
\begin{equation}
\big[X,Y\big](x,y)=2x\frac{\phi(x,y)}{x^2+y^2}\,X-\frac{2y}{x^2+y^2}\,Y
\end{equation}
One can show that the function $(x,y)\mapsto 2x\frac{\phi(x,y)}{x^2+y^2}$ is smooth at the origin, but that the function $(x,y)\mapsto\frac{2y}{x^2+y^2}$ does not admit a limit in $(0,0)$. Hence it is not a smooth function, and the commutator $[X,Y]$ does not take values in $F$, that is: $F$ is not involutive. This example can be used to show that Theorem 1.40 in \cite{Olver} is wrong as well, since the finite set of vector fields consisting of $X$ and $Y$ is not in involution, even though the induced distribution is integrable. \end{example}

Very interestingly, Hermann was not very acknowledged for his third statement, even though it had deep consequences regarding the integrability issues. Every mathematician that tried to prove some result on integrability of smooth distributions systematically emphasized the importance of working with the integral curves of the family of vector fields $F$, in particular to show that the distribution is $F$-invariant along the integral paths. Being the first to bring the attention to this idea, it seems natural to emphasize Hermann's work on integration of generalized distribution as one of the most influential of the field. 

As a final remark, notice that the fact that a smooth distribution is integrable does not necessarily imply that the sheaf of its sections is locally finitely generated (see the final remark in Section \ref{background}).
In the years following the breakthrough of Hermann and Nagano, some attempts were made to find the minimal assumptions that are sufficient for a smooth distribution to be integrable. For example, after a careful analysis of the point in Nagano's proof that requires analyticity, Matsuda provided an adaptation of Nagano's theorem to the smooth case, at the price of requiring rather unnatural conditions \cite{matsuda1968}. 

\section{Improvements and achievements}\label{lobry}

After Hermann, the first important contribution to the problem of integrating smooth distributions was made by Lobry in 1970 \cite{Lobry70}. 
He tried to reproduce the proof of Hermann by weakening the two assumptions that $F$ is involutive and locally finitely generated, and mixing them in a unique condition. Thus, Lobry proposed the following: a set of vector fields $F$ is \emph{locally of finite type} if for every $x\in M$ there exist $X_1,\ldots, X_p\in F$ that span $\mathcal{D}^F_x$, and such that for every $X\in F$, there exists an open neighborhood $U$ of the point $x$ and some functions $(f_{ij})_{1\leq i,j\leq p}\in\mathcal{C}^\infty(U)$ such that:
\begin{equation}\label{lobrysconditions}
[X,X_i](y)=\sum_{j=1}^p\,f_{ij}(y)\, X_j(y)
\end{equation}
for every $y\in U$. Originally Lobry did not require that the subset of vector fields $X_1,\ldots,X_p$ span the distribution at $x$, and it was a condition that Sussmann thought was missing so he added it in his paper when he referred to Lobry's conditions \cite{Sussmannofficiel}. 

The main difference between Hermann's and Lobry's conditions \eqref{lobrysconditions} is that in the first case, since one can pick up a set of local generators of the family $F$, they span the distribution $\mathcal{D}^F$ in some neighborhood of $x$, whereas in Lobry's assumption, the set of vector fields $F'$ that span $\mathcal{D}^F_x$ may not span $\mathcal{D}^F$ in any neighborhood of $x$. What can only be shown is that the distribution $\mathcal{D}^{F'}$ has constant rank on the integral curve of any elements $X\in F$ if we are sufficiently close to $x$. Hence, using the same kind of arguments as in the proof of Hermann is not sufficient to conclude that the distribution $\mathcal{D}^F$ has constant rank on the integral curves of elements of $F$. 
Thus, contrary to the claim of Lobry, the condition that a family of vector fields is locally of the finite type is not a sufficient condition for integrability. This was first noted by Stefan who proposed more subtle conditions in \cite{stefan1980}, see Sections \ref{stefan} and  \ref{conclusion}.

 On the other hand, it may happen that Lobry's conditions \eqref{lobrysconditions} are sufficient for integrability, when applied to the correct set of vector fields. In particular, a smooth distribution may be integrable if the sheaf of sections of $\mathcal{D}$ satisfies Lobry's conditions. This is a claim made by Stefan in 1974 \cite{Stefanofficiel}, but he did not provided any proof before his 1980 paper, as a corollary of a more general proposition, see Theorem 4 in \cite{stefan1980}. We will show in Section \ref{stefan} that even though Theorem 4  is wrong as it is written in \cite{stefan1980}, it can be subtly modified to obtain a correct proof of Stefan's claim on Lobry's conditions.

 Sussmann himself, convinced of the validity of Lobry's conditions in broad generality and of the truthfulness of the proof of Lemma 1.2.1 in \cite{Lobry70}, provided a refinement of the condition that $F$ is locally of the finite type by noticing that since one works on the integral curves of elements of $F$, one can get rid of the open neighborhood condition and only ask that the bracket $[X,X_i]$ is defined on the integral curve on $X$. In other words, Sussmann's integrability conditions are that for every $x\in M$ there exist $X_1,\ldots, X_p\in F$ that span $\mathcal{D}^F_x$, and such that for every $X\in F$, there exists some $\epsilon>0$ and some functions $(g_{ij})_{1\leq i,j\leq p}\in\mathcal{C}^\infty\big(\left]-\epsilon,\epsilon\right[\big)$ such that:
\begin{equation}\label{sussmannsconditions}
\big[X,X_i\big]\big(\phi^X_t(x)\big)=\sum_{j=1}^p\,g_{ij}(t)\, X_j\big(\phi^X_t(x)\big)
\end{equation} 
 for every $t\in\left]-\epsilon,\epsilon\right[$.
 Unfortunately, even if this last condition seems mathematically satisfying because it appears as an optimized generalization of Hermann's condition for integrability, it is not sufficient. This was pointed out by Balan in \cite{balan}.
 
 Independently, Stefan, in his 1974 paper \cite{Stefanofficiel} (however written and submitted in 1973), provided a resembling condition that was sufficient for integrability. He slightly modified the wording and added the conditions that the vector fields $X_1,\ldots,X_p$ depend on the choice of the vector field $X\in F$, and that they span $\mathcal{D}^F$ on the integral curve of $X$. In other words, for every $x\in M$ and $X\in F$, there exists a finite set of vector fields $X_1,\ldots,X_p\in F$, some $\epsilon >0$ and some functions $(g_{ij})_{1\leq i,j\leq p}\in\mathcal{C}^\infty\big(\left]-\epsilon,\epsilon\right[\big)$ such that:
 \begin{enumerate}
 \item $\mathcal{D}^F_{\phi^X_t(x)}=\mathrm{Span}\Big(X_1\big(\phi^X_t(x)\big),\ldots,X_p\big(\phi^X_t(x)\big)\Big)$
 \item $\big[X,X_i\big]\big(\phi^X_t(x)\big)=\sum_{j=1}^p\,g_{ij}(t)\, X_j\big(\phi^X_t(x)\big)$
 \end{enumerate}
 for every $t\in\left]-\epsilon,\epsilon\right[$. The important idea is that now the vector fields $X_1,\ldots, X_p$ depend both on $x$ and on $X$. The proof that a family of vector fields satisfying such conditions induce an integrable distribution follows the exact same lines as Hermann's proof. 
 However, since it is usually very cumbersome to check Stefan's integrability conditions, mathematicians do not use them and they are today mostly forgotten. 
 
 \bigskip
  
 The story does not stop here: the fact that Lobry proposed a wrong claim was systematically emphasized by Stefan \cite{Stefanofficiel, stefan1980}. However, he did not present any counter-example before his 1980 paper \cite{stefan1980}, nor did he ever publicly mentioned that Sussmann's conditions were not sufficient either, even though he may have been completely aware of it. This observation was made by Balan in 1994 \cite{balan}. He explained in details that the implication $(e)\Longrightarrow(d)$ of Theorem 4.2 in \cite{Sussmannofficiel} (that relies on Sussmann's conditions \eqref{sussmannsconditions}) is false, but that the equivalences $(a)\Longleftrightarrow(b)\Longleftrightarrow(c)\Longleftrightarrow(d)\Longleftrightarrow(f)$ are true (and these form the content of the so called `Stefan-Sussmann Theorem').  The counter argument that was proposed by Stefan to refute Lobry's claim -- and that also works to refute Sussman's conditions \eqref{sussmannsconditions} -- is the following:
 \begin{example} \label{Arjen}
Let $M=\mathbb{R}^2$ and let $F$ be the family of vector fields containing all the vector fields of the form:
\begin{equation}
f(x,y)\,\frac{\partial}{\partial x}+g(x,y)\,\frac{\partial}{\partial y}
\end{equation}
for some functions $f,g\in\mathcal{C}^\infty(\mathbb{R}^2)$ such that $g\equiv0$ in some neighborhood of $(0,0)$. The family $F$ is actually a $\mathcal{C}^\infty(\mathbb{R}^2)$-module, that is locally of finite type \cite{stefan1980, balan}. However the induced distribution $\mathcal{D}^F$ turns out to be:
\begin{equation*}
\mathcal{D}_{(x,y)}=
\begin{cases}T_{(x,y)}\mathbb{R}^{2} \hspace{0.8cm} \text{for}\quad (x,y)\neq(0,0)\\
\mathrm{Span}\big(\frac{\partial}{\partial x}\big) \hspace{0.55cm} \text{for}\quad (x,y)=(0,0)
\end{cases}
\end{equation*}
which is obviously not integrable at the origin. Stefan noticed that the space of sections of $\mathcal{D}^F$ was not of the finite type, though. That is why he conjectured in 1974 that if the space of sections of a distribution is locally of finite type, then the distribution is integrable \cite{Stefanofficiel}. 
 \end{example}
Following Stefan refutations, Lobry published a public erratum  \cite{Lobry70erratum}. Even if he was wrong on this precise point, he nonetheless has to be acknowledged for the insight that led to the breakthrough of Stefan and Sussmann. Lobry was indeed the first to show that the flows of the vector fields are a crucial tool to prove integrability.  More precisely, his lemma 1.2.1 in \cite{Lobry70} which he mistakenly attributed to Hermann, and which was later refuted by Stefan \cite{Stefanofficiel}, was implicitly providing the condition for the distribution $\mathcal{D}^F$ to be integrable: it has to be $F$-invariant. 

This is precisely the content of the theorems of Stefan and Sussmann in their subsequent papers \cite{Stefanofficiel, Sussmannofficiel} that were submitted independently in 1972 and 1973, respectively. 
Both Stefan and Sussmann brought the discussion to another level because they did not rely on Lie algebras of globally defined  vector fields anymore as in Hermann's and Nagano's papers, but they allowed $F$ to be a mere family of vector fields that may be locally defined. Their `tour de force' was then to circumvent Chevalley's construction of a refined topology adapted to the integral manifolds of $\mathcal{D}^F$. With different notations, Stefan and Sussmann presented similar results that could be reformulated as follows:
\begin{theoreme}
\emph{\textbf{Stefan-Sussmann (1973)}} \label{theosussmann}
Let $M$ be a smooth manifold and let $\mathcal{D}$ be a smooth distribution. Then $\mathcal{D}$ is integrable if and only if it is generated by a family $F$ of smooth vector fields, and is invariant with respect to $F$.
\end{theoreme}


Historically, it was Sussmann who first published this result in a short note without proofs in January 1973 in the Bulletins of the American Mathematical Society \cite{sussmanntrailer}. During the same year, his seminal paper was published in June in the Transactions of the American Mathematical Society \cite{Sussmannofficiel}. Both were submitted in June 1972. On the other hand, Stefan submitted his own article in July 1973 to the Proceedings of the London Mathematical Society \cite{Stefanofficiel}, but it was only published in December 1974. To claim his result a little bit faster, he submitted a short note to the Bulletins of the American Mathematical Society in March 1974 \cite{stefantrailer}, that was actually published in November 1974. These two papers are a condensate of the work he has done during his PhD, that he defended in December 1973 at the University of Warwick \cite{stefanphd}. However, it seems that Stefan was not aware of Sussmann's work before June 1973, as he says explicitely in \cite{stefantrailer}, and as he emphasizes in the introduction of his paper \cite{Stefanofficiel} that the draft was already written when he heard about Sussmann's papers. This, the deep understanding of Stefan on the questions of integrability, together with the dissemblance of Stefan's and Sussmann's notations and formalism, exclude any suspicion of plagiarism.

\begin{example}
To illustrate this theorem, let us go back to Example \ref{ex:nagano}, where a non-integrable distribution has been presented. Indeed, let $F$ be any family of vector fields that generates $\mathcal{D}$, and let $u=(0,y)$ be any point of the vertical axis. Then by definition we know that there is a vector field $X\in F$ which is defined in a neighborhood $U$ of $u$ and such that $X(u)=\frac{\partial}{\partial x}$. Then one can always push forward the distribution $\mathcal{D}_v=T_v\mathbb{R}^2$, for some $v\in U\cap \{(x,y)\,|\,x>0\}$, to the left half-space, using the flow of $-X$. But on the left half-space, the distribution is one-dimensional, hence  $\mathcal{D}$ is indeed not $F$-invariant. As expected it is not integrable either.
\end{example}

The formulation of Theorem \ref{theosussmann} is rather satisfying because it is a direct analogue of Frobenius' theorem, having the advantage of making evident the condition for integrability in terms of families of vector fields. 
In the case that the distribution $\mathcal{D}^F$ is not integrable, the $F$-orbits still do exist and are submanifolds or $M$. Stefan and Sussmann characterized the tangent space of these orbits as the smallest distribution containing $\mathcal{D}^F$ and that is $F$-invariant, see Theorem 4.1 in \cite{Sussmannofficiel} and Theorem 1 in \cite{stefan1980}, where Stefan finally adopted Sussmann's notations. This is the content of the well known \emph{Orbit Theorem} in modern day control theory \cite{jurdjevic, bullo, Agrachev}. In this field, the original statements of Stefan and Sussmann have been slightly modified to obtain a more convenient formulation adapted to control theorists' needs. The Orbit Theorem is usually attributed to Nagano and Hermann, or Nagano and Sussmann -- because control theorists often work in the analytic context. Then they usually drop the $F$-invariance which is no longer necessary in the analytic case, in favor of involutivity. This could be a bit confusing for someone who is exterior to the field. Another important remark is that the original articles of Stefan and Sussmann are very difficult to go through, either because the notations are unusual (in Stefan \cite{Stefanofficiel}), or because the proof is very tedious (in Sussman \cite{Sussmannofficiel}). For all these reasons, the result of Stefan and Sussmann has not been fully acknowledged, adding more confusion for those who are not specialists of the field.

\section{Further developments}\label{stefan}

After the publication of the groundbreaking results of Stefan and Sussmann, the research on the problem of integrating distributions into foliations did not stop. Stefan actually presented the deepest understanding of the questions of integrability, that is why he proposed other approaches to the problem.  In particular, in 1980 \cite{stefan1980} he introduced the notion of local subintegrability: we say that a set of vector fields $F$ is \emph{locally subintegrable} if for every $x\in M$ there exists a finite subset $F'\subset F$ and an open neighborhood $U$ of $x$ such that:
\begin{enumerate}
\item $\mathcal{D}^{F'}_x=\mathcal{D}^F_x$,
\item $\mathcal{D}^{F'}$ is integrable on $U$,
\item for every $X\in F$, there exists $\epsilon>0$ such that $\big(\phi^X_t\big)_*\big(\mathcal{D}^{F'}_x\big)=\mathcal{D}^{F'}_{\phi^X_t(x)}$ for every $|t|<\epsilon$.
\end{enumerate}
Then, in Theorem 4 in \cite{stefan1980}, Stefan claims that, given a family of vector fields $F$, the distribution $\mathcal{D}^F$ is integrable if and only if the $\mathcal{C}^\infty(M)$-module generated by $F$, that we denote by $F^\#$, is locally subintegrable. It turns out that this claim is wrong, as Example \ref{Arjen} refutes it. This was first noted by Balan in 1994 in \cite{balan}, where he pointed out that even if Lemma 6.1 in \cite{stefan1980} is always true, Lemma 6.2 is wrong, hence implying that Theorem 4 is wrong. A careful analysis shows however that the mistake described by Balan  in Lemma 6.2 vanishes if, instead of $F^\#$, one considers the sheaf $\Gamma\big(\mathcal{D}^F\big)$ of smooth sections of $\mathcal{D}^F$. 
Under this condition, the proof of Theorem 4 in \cite{stefan1980} is true, implying the result that, given a family of vector fields $F$, the induced distribution $\mathcal{D}^F$ is integrable if $\Gamma(\mathcal{D}^F)$ is locally subintegrable.

Stefan proved additionally in Proposition 6.3 in \cite{stefan1980} that his local subintegrability conditions could be written with a formulation that is very close to Lobry's and Sussmann's conditions, Equations \eqref{lobrysconditions} and \eqref{sussmannsconditions} respectively. With all these results combined together, we have:

\begin{theoreme}\emph{\textbf{Stefan (1980)}}\label{Stefantheo}
Let $M$ be a smooth manifold and let $\mathcal{D}$ be a smooth distribution on $M$. Then the following conditions are equivalent:
\begin{enumerate}
\item $\mathcal{D}$ is integrable,
\item $\Gamma(\mathcal{D})$ is locally subintegrable,
\item for every $x\in M$, there exist a family of vector fields $X_1,\ldots,X_p\in\Gamma(\mathcal{D})$ defined on an open neighborhood $U$ of $x$ such that:
\begin{itemize}
\item $\mathcal{D}_x=\mathrm{Span}\big(X_1,\ldots,X_p\big)$,
\item there exist smooth functions $f_{ij}{}^k\in\mathcal{C}^\infty(U)$ such that:
\begin{equation}
[X_i,X_j](y)=\sum_{1\leq k\leq p}\ f_{ij}{}^k(y)\, X_k(y),
\end{equation}
for every $y\in U$, and every $1\leq i,j\leq p$,
\item for every $X\in\Gamma(\mathcal{D})$, there exists $\epsilon>0$ and functions $g_{ij}\in\mathcal{C}^\infty\big(\left]-\epsilon,\epsilon\right[\big)$ such that:
\begin{equation}
\big[X,X_i\big]\big(\phi^X_t(x)\big)=\sum_{j=1}^p\,g_{ij}(t)\, X_j\big(\phi^X_t(x)\big)
\end{equation}
for every $t\in\left]-\epsilon,\epsilon\right[$.
\end{itemize}
\end{enumerate}
\end{theoreme}

It is easy to show that 1.  implies 2. and 3. The equivalence $2.\Longleftrightarrow3.$ is Proposition 6.3 in \cite{stefan1980}, whereas the implication $2.\Longrightarrow1.$ comes from Theorem 4 in \cite{stefan1980}, when applied to $\Gamma(\mathcal{D})$. Notice that Lobry's conditions \eqref{lobrysconditions} -- when applied to $\Gamma(\mathcal{D})$ -- imply item 3. This proves that Lobry's conditions are sufficient conditions for integrability, if one consider the space of sections of the distribution, and not any generating set of vector fields. Interestingly, this corollary was stated by Stefan in his 1974 paper \cite{Stefanofficiel}, but he nevertheless chose the $\mathcal{C}^\infty(M)$-module $F^\#$ to prove Theorem 4 in his 1980 paper \cite{stefan1980}, which was not a conclusive choice. Notice also that item 3. finally provides the local description of integrable smooth distributions that were sought for years, through the work of Hermann, Lobry, and Sussmann. Item 3. is a condition that is stronger that Sussmann's but weaker than Lobry's.

\bigskip
Independently of this discussion, Balan 
  proposed an alternative formulation of Stefan's local subintegrability conditions that would be valid for any $\mathcal{C}^\infty(M)$-module of vector fields \cite{balan}. In the following the family $F$ will hence be considered as carrying a $\mathcal{C}^\infty(M)$-module structure. He understood that the flaw of the condition that a family  $F\subset\mathfrak{X}(M)$ be locally of finite type is that, given a vector field $X\in F$, the open set $U$ on which the bracket with $X$ is defined depends on $X$. The same argument applies to Sussmann's conditions \eqref{sussmannsconditions}, where $\epsilon$ depends on $X$. This is precisely the reason why Example \ref{Arjen} works so well.  

That is why Balan proposed to modify Sussmann's conditions into a stronger one, where $\epsilon$ does not depend on the choice of the vector field $X\in F$. Balan encoded this condition not directly with a parameter $\epsilon$, but with an indirect way, with the help of some open subset.
More precisely Balan's integrability conditions can be stated as follows: for any $x\in M$, there exist $X_1,\ldots,X_p\in F$ and some open set $U\subset M$ such that 
$X_1(x),\ldots,X_p(x)$ is a basis of $\mathcal{D}^F_x$, and that for every $X\in F$, there exist smooth functions $(g_{ij})_{1\leq i,j\leq p}\in\mathcal{C}^\infty\big(\left]-\mu_X,\mu_X\right[\big)$ such that:
\begin{equation}\label{Balan}
\big[X,X_i\big]\big(\phi^X_t(x)\big)=\sum_{j=1}^p\,g_{ij}(t)\, X_j\big(\phi^X_t(x)\big)
\end{equation}
for every $t\in\left]-\mu_X,\mu_X\right[$, where $\mu_X=\mathrm{sup}\big\{s\,|\,\phi^X_t(x)\in U \text{ for all }|t|<s\,\big\}$. Then, in Theorem 2.1 in \cite{balan}, he claims that the distribution $\mathcal{D}^F$ is integrable if and only if the $\mathcal{C}^\infty(M)$-module $F$ satisfies his conditions.

First a few remarks. The formulation of Balan is a bit overdetermined (there is an additional parameter $\epsilon$ in the original formulation of Theorem 2.1 that does not play any role). The main point of his conditions is that for any $X\in F$, the bracket $[X,X_i]$ satisfies Equation \eqref{Balan} on the entire integral curve of $X$ that is contained in $U$. In other words, Balan does not want that Equation \eqref{Balan} be only satisfied on a small part of the integral curve around $x$ that depends on the choice of the vector field $X$. Thus, the role of the open set $U$ in Balan's conditions is to enforce that Equation \eqref{Balan} is satisfied on all the integral curves passing though $x$ and that are defined on a small neighborhood of $x$. Also, contrary to Stefan's integrability conditions presented in Theorem \ref{Stefantheo} where the entire space of sections of the distribution $\mathcal{D}$ had to be taken into account, here we only consider a family of vector fields $F$, as in Theorem \ref{theosussmann}.
 
 Thus Balan's theorem 2.1 in \cite{balan} can be stated as follows:
 
\begin{theoreme}\emph{\textbf{Balan (1994)}}
Let $M$ be a smooth manifold, and let $F\subset\mathfrak{X}(M)$ be a $\mathcal{C}^\infty(M)$-module of vector fields. Then the induced distribution $\mathcal{D}^F$ is integrable if and only if $F$ satisfies Balan's integrability conditions \eqref{Balan}.
\end{theoreme}
\begin{proof} The proof  is essentially an adaptation of Nagano's proof of integrability \cite{Nagano}, using the same splitting $F(x)\oplus\mathcal{G}(x)$ of the family $F$ at the point $x$. Let $B_\epsilon$ be some ball of radius $\epsilon$ centered at the origin in $T(x)$, then the exponential map $\mathrm{exp}:B_\epsilon\to M$ defines an embedded submanifold $N_x$ of $M$. The core of Nagano's proof is that $\big[F(x),F(x)\big]$ is zero on $N_x$. Since $\big[F(x),F(x)\big]\subset\mathcal{G}(x)$, the result is proven by showing that $\mathcal{G}(x)$ vanishes on $N_x$, which is done by using the analyticity of the objects. In Balan's paper, one has to use a different argument since the objects are not analytic anymore. To do this, Balan singles out the condition of integrability in Lemma 3.4 in \cite{balan}. In Nagano's proof, item 2. is a consequence of analyticity, and item 1. is a consequence of item 2. In the smooth case, Balan uses his conditions to show item 1., which is then necessary to prove item 2. But this last item is not necessary to prove that $N_x$ is an integral manifold, since Equation (1.5) in Nagano's proof only require the validity of item 1.
To show the first item, Balan assumes without loss of generality that there exists a set of vector fields that satisfies his conditions and also that induces a splitting in the sense of Nagano. He then uses this particular basis to show both the first item and then the second item of Lemma 3.4 in \cite{balan}. The notations used to describe the differential equation are not very clear, and one can find a clearer presentation in Lemma 6.1 in Sussmann's paper \cite{Sussmannofficiel}. Finally, the second item of Lemma 3.4 in \cite{balan} is necessary to prove the `only if' part of the theorem, that is: if $\mathcal{D}^F$ is integrable, then Balan's conditions \eqref{Balan} are satisfied. \end{proof}

\section{Conclusion}\label{conclusion}

We have shown in the preceding three sections that the road to a definitive answer on the issue  of integrating distributions reveals itself incredibly flourishing and twisting from the 1970s on. There are many results that turned to be wrong, and many claims whose proof are inconclusive. To this day, the main theorems that have been proven are: Nagano's theorem and Hermann's theorem in Section \ref{nagano}, Stefan-Sussmann's theorem in Section \ref{lobry}, and Stefan's theorem and Balan's theorem in Section \ref{stefan}. 
It is interesting that they  involve different objects such as sub-Lie algebras of $\mathfrak{X}(M)$, spaces of sections of distributions, $\mathcal{C}^\infty(M)$-modules of vector fields, and even mere families of vector fields in Stefan-Sussmann's theorem. Mathematicians could now choose the theorem that is more adapted to their needs. In particular, in control theory and in geometry, the most popular theorems are Nagano's and Hermann's. 

 Hermann's influence on the field has to be emphasized as the founder of geometric control theory, and as the first one who gave results in the smooth and analytic cases, with a powerful argument: that the integral paths are the objects of interest when one attempts to integrate a distribution. 
In the same way, Stefan has to be acknowledged for his insights and his understanding of the topic that enabled him to produce many new results on the question of integrability, most notable the local characterization of integrable smooth distributions. This is all the more important, since he had a very short career before his tragic death while climbing mount Tryfan in 1978. 

Another achievement of Stefan is the characterization of the leaves of a singular foliation with respect to the  smooth structure on $M$ \cite{Stefanofficiel}, which is based on the usual definition of regular foliations \cite{Lee}. 
Any regular foliation is characterized by a \emph{foliated atlas}, which means that any of its charts is a saturated set: it is the union of disjoint connected submanifolds of a specific form that we call \emph{plaques}. In the singular case, the definition is slightly modified because transition functions cannot be defined in the same way as in the regular case. More precisely, given a smooth manifold $M$ and a foliation on it, we say that $M$ is equipped with a \emph{distinguished atlas} if for any point $x\in M$, there exist an open neighborhood $U$ and a diffeomorphism $\varphi:U\to V\times W$, such that:

\begin{enumerate}
\item $V\subset\mathbb{R}^p$ and $W\subset\mathbb{R}^{n-p}$ are open sets, where $p=\mathrm{dim}(L_x)$,
\item $\varphi(x)=(0,0)$,
\item for any leaf $L$, we have $\varphi(L\cap U)=V\times l_L$, with $l_L=\big\{y\in W\,|\,\varphi^{-1}(0,y)\subset L\big\}$.
\end{enumerate}
 Stefan made a definitive progress in foliation theory by showing that the orbits of a family of vector fields are indeed leaves of a foliation that satisfies his criteria \cite{Stefanofficiel}.


Now addressing the persistent claim that Balan allegedly corrected Stefan-Sussmann's theorem, we saw in Section \ref{lobry} that Stefan created the counter-example \ref{Arjen} to refute Lobry's conditions of integrability, and in Section \ref{stefan}, we saw that Balan used this counter-example to refute both the claim that Sussmann's conditions are sufficient for integrability, and the claim that Stefan's local integrability is a sufficient condition as well. However we have seen that local integrability becomes a sufficient condition when one considers the space of sections of the distribution. Then Theorem 4 in \cite{stefan1980} becomes Theorem \ref{Stefantheo} above. Moreover, Stefan-Sussmann's theorem, as stated in Section \ref{lobry} and in most books of geometry and of control theory, is not impacted by Balan's observations. In view of all these arguments, there is no reason to further propagate the idea that Stefan-Sussmann's theorem is incomplete or even wrong, and that Balan corrected it. What is true is that Balan proved that two specific assumptions in two different papers were not sufficient for integrability, and provided his own view on the problem by proposing a new statement on integrability.

\bigskip
It is now time to turn to present-day geometry and talk about conventions. An important point is that a distribution in itself does not carry much informations because it can be generated by several sets of vector fields. A family of vector fields indeed carries more data, as can be shown on the following example:

\begin{example}The Lie algebras $\mathfrak{gl}(2)$, $\mathfrak{sl}(2)$ and $\mathbb{C}^*$ (seen as a real matrix algebra) act on $\mathbb{R}^2$ via their respective  different actions, but they induce the same integrable distribution. The corresponding foliation has two leaves: the point at the origin, and the punctured plane. It is shown in \cite{AndrouSkandal} that the holonomy groupoid corresponding to these various actions are drastically different.
\end{example}

 This example shows that a family of vector fields contains more informations than the distribution that it induces. The focus on the family of vector fields rather than on the distribution draws a link with the original motivation of geometric control theory: solving linear differential systems using tools from geometry, and considering that the vector fields are the main objects of interests. This is not a new idea because Nagano himself for example defines a linear differential system as the $\mathcal{C}^\infty(M)$-module generated by the sub-Lie algebra $F\subset\mathfrak{X}(M)$ \cite{Nagano}.

There has also been a shift in the kind of object that are manipulated. Today, some geometers are more accustomed to manipulate  modules or sheaves of vector fields than just families of vector fields as was typical of Stefan's and Sussmann's work.
In the field of Poisson geometry for example, there are various notions and definitions but they all rely on this module property. A sub-module of compactly supported vector fields that is locally finitely generated and involutive is called a \emph{singular foliation} in \cite{AndrouSkandal}, or a \emph{Stefan-Sussmann foliation} in \cite{AndrouZamb4}. A different formulation appears in \cite{thesis}: a \emph{Hermann foliation} is a sub-sheaf $\mathcal{F}:U\mapsto\mathcal{F}(U)$ of the sheaf of vector fields $\mathfrak{X}$ that is locally finitely generated and closed under Lie bracket. Here, we say that a sheaf $\mathcal{F}$ is \emph{locally finitely generated} if for any $x\in M$, there exists an open set $U$ such that $\mathcal{F}(U)$ is finitely generated as a $\mathcal{C}^\infty(U)$-module.

It has been shown that these two different notions are in one-to-one correspondence \cite{roy}.
Thus, it would be useful to find a common denomination for these objects that are equivalent, but bear different names. In any case Hermann's theorem implies that the distributions induced by either `Stefan-Sussmann foliations' or `Hermann foliations' are integrable. There is no need to use Stefan-Sussmann's theorem to show this result. As a historical note, in Hermann's original paper \cite{Hermann1962}, the sub-Lie algebras $F\subset\mathfrak{X}(M)$ that he is studying are called \emph{foliations with singularities}. Hence that would justify that one uses the term \emph{singular foliations} for the equivalent notions used in \cite{AndrouSkandal, AndrouZamb4, thesis}, as this term was originally used by Hermann to precisely designate those families of vector fields that are locally finitely generated and involutive.


\bibliographystyle{plainyr}
\bibliography{mabibliographie}

\begin{thebibliography}{10}

\bibitem{Deahna}
F.~Deahna.
\newblock {Ueber die Bedingungen der Integrabilit\"at line\"arer
  Differentialgleichungen erster Ordnung zwischen einer beliebigen Anzahl
  ver\"anderlicher Gr\"o{\ss}en}.
\newblock {\em J. Reine Andew. Math.}, 20:340--349, 1840.

\bibitem{Clebsch}
A.~Clebsch.
\newblock {Ueber das Pfaffsche Problem}.
\newblock {\em J. Reine Andew. Math.}, 60:193--251, 1861.

\bibitem{Frobenius}
F.~Frobenius.
\newblock {Ueber das Pfaffsche Problem}.
\newblock {\em J. Reine Andew. Math.}, 82:230--315, 1877.

\bibitem{Hermann1962}
R.~Hermann.
\newblock The differential geometry of foliations, ii.
\newblock {\em J. Appl. Math. Mech.}, 11:303--315, 1962.

\bibitem{Hermann1963}
R.~Hermann.
\newblock {On the Accessibility Problem in Control Theory}.
\newblock In {\em {International Symposium on Nonlinear Differential Equations
  and Nonlinear Mechanics}}, pages 325--332. Academic Press, New York, 1963.

\bibitem{Nagano}
T.~Nagano.
\newblock Linear differential systems with singularities and an application to
  transitive lie algebras.
\newblock {\em J. Math. Soc. Japan}, 18(4):398--404, 1966.

\bibitem{matsuda1968}
M.~Matsuda.
\newblock An integration theorem for completely integrable systems with
  singularities.
\newblock {\em Osaka J. Math.}, 5(2):279--283, 1968.

\bibitem{Lobry70}
C.~Lobry.
\newblock Contr\^olabilit\'e des syst\`emes non lin\'eaires.
\newblock {\em SIAM J. Control Optim.}, 8(4):573--605, 1970.

\bibitem{stefanphd}
P.~{Stefan}.
\newblock {\em {Accessibility and singular foliations}}.
\newblock PhD thesis, University of Warwick, dec 1973.

\bibitem{Sussmannofficiel}
H.~Sussmann.
\newblock Orbits of families of vector fields and integrability of
  distributions.
\newblock {\em Trans. Amer. Math. Soc.}, 180:171--188, 1973.

\bibitem{sussmanntrailer}
H.~Sussmann.
\newblock Orbits of families of vector fields and integrability of systems with
  singularities.
\newblock {\em Bull. Amer. Math. Soc.}, 79(1):197--199, 1973.

\bibitem{stefantrailer}
P.~Stefan.
\newblock Accessibility and foliations with singularities.
\newblock {\em Bull. Amer. Math. Soc.}, 80(6):1142--1145, 1974.

\bibitem{Stefanofficiel}
P.~Stefan.
\newblock Accessible sets, orbits, and foliations with singularities.
\newblock {\em Proc. London Math. Soc.}, s3-29(4):699--713, 1974.

\bibitem{Lobry70erratum}
C.~Lobry.
\newblock Erratum: Contr\^olabilit\'e des syst\`emes non lin\'eaires.
\newblock {\em SIAM J. Control Optim.}, 14(2):387--387, 1976.

\bibitem{stefan1980}
P.~Stefan.
\newblock Integrability of systems of vectorfields.
\newblock {\em J. London Math. Soc.}, s2-21(3):544--556, 1980.

\bibitem{balan}
R.~Balan.
\newblock Note about integrability of distributions with singularities.
\newblock {\em Boll. Un. Mat. Ital.}, 7:335--344, 1994.

\bibitem{AndrouSkandal}
I.~Androulidakis and G.~Skandalis.
\newblock {The holonomy groupoid of a singular foliation}.
\newblock {\em J. Reine Andew. Math.}, 626:1--37, 2009.

\bibitem{texas}
L.~D. {Drager}, J.~M. {Lee}, E.~{Park}, and K.~{Richardson}.
\newblock {Smooth distributions are finitely generated}.
\newblock {\em Ann. Global Anal. Geom.}, 41(3):357--369, 2012.

\bibitem{AndrouZamb4}
I.~Androulidakis and M.~Zambon.
\newblock {Stefan-Sussmann singular foliations, singular subalgebroids, and
  their associated sheaves}.
\newblock {\em Int. J. Geom. Methods Mod. Phys.}, 13(Supp. 1):1641001, 2016.

\bibitem{thesis}
S.~Lavau.
\newblock {\em {Lie $\infty$-algebroides et feuilletages singuliers}}.
\newblock PhD thesis, Universit\'e Claude Bernard - Lyon~1, November 2016.

\bibitem{roy}
R.~{Wang}.
\newblock {\em {On integrable systems and rigidity for PDEs with symmetry}}.
\newblock PhD thesis, Utrecht University, September 2017.

\bibitem{Agrachev}
A.~Agrachev and Y.~Sachkov.
\newblock {\em Control Theory from the Geometric Viewpoint}, volume~84 of {\em
  Encyclopaedia of Mathematical Sciences}.
\newblock Springer-Verlag, Berlin, Heidelberg, 2004.

\bibitem{bullo}
F.~Bullo and A.~Lewis.
\newblock {\em Geometric Control of Mechanical Systems}, volume~49 of {\em
  Texts in Applied Mathematics}.
\newblock Springer-Verlag, New York, 2003.

\bibitem{chevalley}
C.~Chevalley.
\newblock {\em Theory of Lie Groups. I}, volume~8 of {\em Princeton
  Mathematical Series}.
\newblock Princeton University Press, Princeton, 1946.

\bibitem{Hawkins}
T.~Hawkins.
\newblock {\em The Mathematics of Frobenius in Context}.
\newblock Springer-Verlag, New York, 2013.

\bibitem{jurdjevic}
V.~Jurdjevic.
\newblock {\em Geometric Control Theory}, volume~52 of {\em Cambridge Studies
  in Advanced Mathematics}.
\newblock Cambridge University Press, Cambridge, New York, 1996.

\bibitem{Lee}
J.~M. Lee.
\newblock {\em Introduction to Smooth Manifolds}, volume 218 of {\em Graduate
  Texts in Mathematics}.
\newblock Springer-Verlag, New York, 2002.

\bibitem{Olver}
P.~Olver.
\newblock {\em Applications of Lie Groups to Differential Equations}, volume
  107 of {\em Graduate Texts in Mathematics}.
\newblock Springer-Verlag, New York, 1986.

\end{thebibliography}

\end{document}